\documentclass[a4paper,11pt,reqno]{amsart}
\usepackage{amsfonts}
\usepackage{amsmath}
\usepackage{logicproof}
\usepackage{amssymb, amsthm}
\usepackage{tabularx}
\usepackage{caption}
\usepackage{booktabs}
\usepackage{multirow,dsfont}
\usepackage{array}
\usepackage{centernot}
\usepackage{enumitem}
\newcolumntype{L}[1]{>{\raggedright\let\newline\\\arraybackslash\hspace{0pt}}m{#1}}
\newcolumntype{C}[1]{>{\centering\let\newline\\\arraybackslash\hspace{0pt}}m{#1}}
\newcolumntype{R}[1]{>{\raggedleft\let\newline\\\arraybackslash\hspace{0pt}}m{#1}}
\usepackage[symbol]{footmisc}
\usepackage[labelfont=bf,format=plain,justification=raggedright,singlelinecheck=false]{caption}
\usepackage{mathrsfs}
\usepackage[]{mdframed}
\usepackage[colorlinks]{hyperref}
\usepackage{tcolorbox}
\usepackage{adjustbox}
\tcbuselibrary{theorems}
\newtcbtheorem[number within=section]{mytheo}{Note}%
{colback=white!5,colframe=black!35!black,fonttitle=\bfseries}{th}
\numberwithin{equation}{section}

\usepackage{graphicx}
\usepackage{textgreek}

 {
      \theoremstyle{plain}
      \newtheorem{assumption}{Assumption}
  }
\newtheorem{theorem}{Theorem}[section]
\newtheorem{definition}[theorem]{Definition}
\newtheorem{remark}[theorem]{Remark}
\newtheorem{corollary}[theorem]{Corollary}
\newtheorem{prop}[theorem]{Proposition}
\newtheorem{lemma}[theorem]{Lemma}
\newtheorem{example}[theorem]{Example}

\begin{document}

\title[Polynomial Regression Analysis]{Convergence Analysis of function-on-function Polynomial regression model}

\author[N. Gupta]{Naveen Gupta}
\address[N. Gupta]{Indian Institute of Technology Delhi, India}
\email{ngupta.maths@gmail.com}
\author{S. Sivananthan}
\address[S. Sivananthan]{Indian Institute of Technology Delhi, India}
\email{siva@maths.iitd.ac.in}
%%%%%%%%%%%%%%%%%%%%%%%%%%% Abstract %%%%%%%%%%%%%%%%%%%%%%%%
\begin{abstract}
In this article, we study the convergence behavior of the regularization-based algorithm for solving the polynomial regression model when both input data and responses are from infinite-dimensional Hilbert spaces. We derive convergence rates for estimation and prediction error by employing general (spectral) regularization under a general smoothness condition without imposing any additional conditions on the index function. We also establish lower bounds for any learning algorithm to explain the optimality of our convergence rates.

   % In this article, we consider a polynomial regression analysis model with functional response to construct an estimator for the unknown slope function. We derive convergence rates while considering general source conditions. We provide lower bounds to felicitate our upper bounds to show that our rates are optimal.
\end{abstract}

\maketitle

%%%%%%%%%%%%%%%%%%%%%%%%%%% Section- partition %%%%%%%%%%%%%%%%%%%%%%%%
\section{Introduction}
In this paper, we explore the polynomial regression model in the context where both the predictor and the response are functional data, meaning that they vary continuously over a domain such as time or space. This framework, referred as the function-on-function polynomial regression model, is mathematically given as follows.
\subsection{Polynomial Regression Model:}
Let $S_{1}$ and $S_{2}$ be two compact subsets of $\mathbb{R}$ and $\{X(\omega,s), Y(\omega, t): s \in S_{1}, t \in S_{2}, \omega \in \Omega \}$ are two random processes, then the Polynomial regression model of degree $p$ with functional responses is given as:
\begin{equation}\label{model}
\begin{split}
    Y(\omega, t) = & \beta_{0}^{*}(t) + \int_{S_{1}} \beta_{1}^{*}(t, s_1) X(\omega, s_1) d\mu(s_1) + \int_{S_{1}}\int_{S_{1}} \beta_{2}^{*}(t, s_1, s_2) X(\omega, s_1) X(\omega, s_2) d\mu(s_1) d\mu(s_2)\\
    & + \ldots +\int_{S_{1}} \ldots \int_{S_{1}} \beta_{p}^{*}(t, s_1, s_2, \ldots , s_p) \prod_{i=1}^{p} X(\omega, s_i) d\mu(s_i) + \epsilon(\omega, t).
\end{split}
\end{equation}
Here $\epsilon(\omega, t) \in L^2(\Omega, \mathbb{P}) \otimes L^2(S_{2})$ is a zero mean error term such that $\mathbb{E}\|\epsilon\|_{L^2(S_2)}^2 = \sigma^2 < \infty$ and $\beta^* = (\beta_{0}^{*}, \beta_{1}^{*}, \ldots, \beta_{p}^{*}) \in \mathbb{L}^2 : = L^2(S_{2}) \oplus (\oplus_{l=1}^{p} L^2(S_{2}) \otimes L^2_{l})$ is the unknown slope function. The notation $L^2_{l}$ has been used for $l-$times tensor product of $L^2(S_{1})$, i.e., $ L^2_{l}= \underbrace{L^2(S_{1}) \otimes L^2(S_{1}) \otimes \ldots \otimes L^2(S_{1})}_{l \text{-times}}.$ For a $u = (u_{0}, u_{1}, \ldots, u_{p}) \in \mathbb{L}^2$, the $\mathbb{L}^2-$ norm of $u$ is given as $\|u\|_{\mathbb{L}^2}^2 = \|u_{0}\|^2_{L^2(S_{2})} + \sum_{l=1}^{p} \|u_{l}\|^2_{L^2(S_{2}) \otimes L_{l}^2}$.\\

For the given problem, we propose the general regularization scheme to construct an estimator for $\beta^*$ and establish the convergence rates under the general smoothness condition over the unknown target function $\beta^*$.\\

\noindent
Polynomial regression model provides a natural generalization to the functional linear regression (FLR) model, a commonly studied model to deal with functional data, as it also encompasses non-linear relation between the predictor and output. The FLR model introduced by Ramsay and Dalzell \cite{ramsay1991some}, a cornerstone of functional data analysis (FDA) \cite{ramsay2002afda,ramsaywhendataarefunctions,kokoszka2017,reiss2017,morris2015functional,wang2016functional}, gained its popularity due to two main reasons. First is the advancement of new techniques that allows data collection in the form of functions rather than finite and the inadequacy of traditional regression models from learning theory \cite{cuckerzhou2007learningtheory, AI2022sergei} to effectively handle such functional data.

In the early stages of FDA, the primary approach involved representing the unknown target function using a specific basis to solve the FLR model. For example, a B-spline basis as in \cite{cardot2003spline}, or more popularly the eigenfunctions of the covariance operator \cite{cai2006prediction, hall2007methodology}. This approach was named functional principal component analysis (FPCA). Another approach to solve the FLR model which gained popularity is the method of regularization in a reproducing kernel Hilbert space (RKHS), i.e., an estimator of the target function is constructed by restricting it to an RKHS \cite{ARKHSFORFLR, tonyyuan2012minimax, ZhangFaster2020, balasubramanian2022unified, gupta2024optimal}.
%to While the literature of functional data has grown so fast and big, it is reasonable to ask for a variety of models from which an appropriate model can be chosen for a given problem.
Although the FLR model is widely studied and performs well, prior research \cite{functional2010quadratic, function_on_function_2020_quadratic_regression} has demonstrated that function quadratic regression can lead to significant improvements over the FLR model. Yao and Müller \cite{functional2010quadratic} analyzed the convergence properties of a quadratic regression model with scalar responses. While \cite{function_on_function_2020_quadratic_regression} established the asymptotic convergence of the function-on-function quadratic regression model, it did not provide explicit convergence rates. The polynomial regression model \eqref{model} extends both the FLR and quadratic regression models, offering a flexible framework to accommodate various functional data scenarios by adjusting the parameter $p$. In \cite{polynomial2023regularization}, the authors derived convergence rates for the polynomial regression model with scalar responses within the $L^2$ framework. Unlike the RKHS approach, the $L^2$ framework does not impose restrictions on the hypothesis space while constructing an estimator for the unknown target function.

Previous studies in this area have largely examined these models in the framework of scalar responses against the functional inputs, which can be somewhat restrictive. This limitation is evident in the well-known Canadian weather data example. In this FDA application, daily temperature measurements from 35 locations across Canada, averaged over the years 1960–1994, are used to predict precipitation at a given location. Under a scalar response model, only the log annual precipitation can be estimated. However, with a functional response model, it becomes possible to predict daily precipitation profiles. Inspired by this example and the generalization of the FLR model \cite{polynomial2023regularization}, we investigate the polynomial regression model when input data and response both are functional by employing general spectral regularization to approximate the unknown target function under the general smoothness assumption over the unknown target function.

In the framework of the FLR, the functional response model has been studied in \cite{Lianheng,xiaoxiao2018functional}. Both of these works have focused their attention to prediction error under certain source conditions, which enables them to operate in settings where the associated operator is compact. However, this restriction prevents their analysis from being extended to more general error measures, for example estimation error. Our work advances this line of research by considering a unified error criterion that simultaneously captures both estimation and prediction errors, while remaining within the $L^2$ framework, similar to \cite{polynomial2023regularization}.\\

In contrast to the analysis in \cite{polynomial2023regularization, Lianheng, xiaoxiao2018functional}, the operator arising in our work does not possess the Hilbert–Schmidt property. This fundamental difference in the spectral nature of the operator brings significant challenges in establishing optimal convergence rates. As noted in \cite{polynomial2023regularization}, addressing such challenges in the context of function-on-function polynomial regression remains an interesting and non-trivial problem. More details on this discussion has been given in Section~\ref{ch_5:sec:main_results}. The specific contributions of our work are outlined in detail below.

%In this article, we utilize the spectral regularization algorithm to study the polynomial regression model with functional response. We derive the optimal convergence rates for the most general form of source condition on the target function by removing all the additional conditions over the index function considered in \cite{polynomial2023regularization}.

\noindent
\subsection{Contribution} 
\emph{(i)} We provide an estimator and derive the convergence rates for the function-on-function polynomial regression model by utilizing the general regularization scheme. \vspace{1.5mm}

\noindent
\emph{(ii)} We extend the concept of the general source condition to the function-on-function polynomial regression setting and establish convergence rates under this smoothness assumption. Moreover, we remove all the restrictive assumptions on the index function $\varphi$ -- a continuous, non-decreasing function with $\varphi(0)=0$-- such as operator monotonicity or Lipschitz continuity.
\vspace{1.5mm}

\noindent
\emph{(iii)} To the best of our knowledge, the optimality of convergence rates for polynomial regression model has not been explored till now. We establish the optimality of our convergence rates by showing that the upper bounds obtained in Theorem \ref{ch_5:upper_bound_main_theorem} align with the lower bounds derived in Theorem \ref{ch_5:lower_bound_main_theorem}.

\noindent
\subsection{Notations} $L^2(S)$ denotes the space of all real-valued square-integrable functions defined on $S$. For $f, g \in L^2(S)$, $L^2$ inner product and norm are defined as $\langle f, g \rangle_{L^2(S)} = \int_{S}f(x)g(x)\, dx$ and $\|f\|^2_{L^2(S)}= \langle f, f \rangle_{L^2(S)}$. For two Hilbert spaces $H_{1}, H_{2}$ and an operator $A: H_{1} \to H_{2}$, $\mathcal{R}(A)$ denotes the range of operator $A$ and the operator norm is defined as 
$\|A\|_{H_{1} \to H_{2}} = \sup \{\|Af\|_{H_{2}} | f \in H_{1}, \|f\|_{H_{1}}=1\}.$ For two positive numbers $a$ and $b$, $a \lesssim b$ means $a \leq cb $ for some positive constant $c$.  For positive sequences $(a_{k})_k$ and $(b_{k})_k$, $a_{k} \asymp b_{k}$ means $a_k \lesssim b_k \lesssim a_k$ for all $k$.  For a random variable $W$ with law $P$ and a constant $b$, $W\lesssim_p b$ denotes that for
any $\delta > 0$, there exists a positive constant $c_\delta<\infty$  such that $P(W\le c_\delta b)\ge \delta$.

\noindent
\subsection{Organization}
The structure of this paper unfolds as follows: In Section \ref{ch_5:sec:model_and_preliminaries}, we provide the regularized estimator of the unknown target function for the polynomial regression model when responses are functional and we present the necessary background required for our analysis. Section \ref{ch_5:sec:main_results} starts with listing our assumptions followed by the convergence rates by considering a general error term which covers both the estimation and the prediction error by utilizing the general regularization scheme under a general smoothness condition over the target function. To ensure the optimality of derived convergence rates, we provide with the matching lower bounds in section \ref{ch_5:sec:lower_bounds}.
\section{Model and Preliminaries}
\label{ch_5:sec:model_and_preliminaries}
\noindent
In this section, we define some operators to simplify our model and provide the necessary background of the general regularization scheme. Taking motivation from the scalar response case for the polynomial regression model considered in \cite{polynomial2023regularization}, we define
\begin{equation*}
\begin{split}
    A_{0} & : L^2(S_{2})  \to L^2(\Omega, \mathbb{P}) \otimes L^2(S_{2}),~ \text{ given as } A_{0}u(\cdot) = u(\omega, \cdot )= u(\cdot),\\
    A_{l} & : L^2(S_{2}) \otimes L^2_{l}   \to  L^2(\Omega, \mathbb{P}) \otimes L^2(S_{2}),~\text{ given as }\\ A_{l}u_{l} & = \int_{S_{1}}\ldots\int_{S_{1}} u_{l}(t, s_1, s_2, \ldots , s_l) \prod_{i=1}^{l} X(\omega, s_i) d\mu(s_i) ~ ; ~ 1 \leq l \leq p.
    \end{split}
\end{equation*}
Then the model $(\ref{model})$ can be seen as
\begin{equation*}
    Y(\omega, t) = A_{0}\beta_{0}^{*}(\omega, t) + A_{1} \beta_{1}^{*}(\omega,t)  + \ldots + A_{p} \beta_{p}^{*}(\omega, t) + \epsilon(\omega,t),
\end{equation*}
and further simplification of the model can be given as:
\begin{equation}
\label{modelequation}
    Y(\omega, t) = \mathbb{A}\beta^*(\omega, t) + \epsilon(\omega,t),
\end{equation}
where $\mathbb{A} := (A_{0}, A_{1}, \ldots, A_{p}): \mathbb{L}^2 \to  L^2(\Omega, \mathbb{P}) \otimes L^2(S_{2})$ defined as
\begin{equation*}
    \mathbb{A}u = \sum_{i= 0}^{p} A_{i}u_{i}, \quad \forall~ u = (u_{0}, \ldots, u_{p}) \in \mathbb{L}^2.
\end{equation*}
Refering to equation $(\ref{modelequation})$, we wish to find an estimator for $\beta^*$ by utilizing the fact that minimizer of $\|Y-\mathbb{A}\beta\|_{L^2(\Omega, \mathbb{P})\otimes L^2(S_{2})}$ will satisfy the operator equation $\mathbb{A}^*\mathbb{A}\beta = \mathbb{A}^*Y$, where $\mathbb{A}^* : L^2(\Omega, \mathbb{P}) \otimes L^2(S_{2}) \to \mathbb{L}^2$ is the adjoint of $\mathbb{A}$. The adjoint operator $\mathbb{A}^*$ is given by
\begin{equation*}
    \mathbb{A}^*g = (A_{0}^*g, A_{1}^*g,\ldots, A_{p}^*g),\quad \forall~ g \in L^2(\Omega, \mathbb{P}) \otimes L^2(S_{2}),
\end{equation*}
where the component operators are defined as
\begin{equation*}
\begin{split}
    A_{0}^* g & = \int_{\Omega}g(\omega, t) d\mathbb{P}(\omega)\\
    A_{l}^* g & = \int_{\Omega}g(\omega, t) \prod_{i=1}^{l} X(\omega, s_{i})d\mathbb{P}(\omega), \quad 1\leq l \leq p.
\end{split}
\end{equation*}
Since the distribution $\mathbb{P}$ is not available, the true solution of the operator equation, denoted as $\beta^*$, remains inaccessible. Therefore, our goal is to construct an estimator for $\beta^*$ using the observed empirical data $\{(X_{1}, Y_{1}), (X_{2}, Y_{2}), \dots, (X_{n}, Y_{n})\}$. For these empirical data points, we have
$$Y_{i}(t) = \mathbb{B}_{i}\beta^* + \epsilon_{i}(t),\quad \forall~ 1 \leq i \leq n.$$
Here the operator $\mathbb{B}_{i} := (B_{i,0},B_{i,1},\ldots,B_{i,p}): \mathbb{L}^2 \to L^2(S_{2})$ is given as:
\begin{equation*}
    \mathbb{B}_{i}u(t) = B_{i,0}u_{0}(t) + \sum_{l = 1}^{p} B_{i,l} u_{l}(t),~ \forall u = (u_{0},\ldots, u_{p}) \in \mathbb{L}^2,
\end{equation*}
where for $1 \leq i \leq n$
\begin{equation*}
\begin{split}
    B_{i,0} & : L^2(S_{2}) \to L^2(S_{2}), ~ \text{ given as } B_{i,0}u_{0}(\cdot) = u_{0}(\cdot),\\
    B_{i,l} & : L^2(S_{2}) \otimes L^2_{l}   \to  L^2(S_{2}),~\text{ given as } \\ B_{i,l}u_{l} & = \int_{S_{1}}\ldots\int_{S_{1}} u_{l}(t, s_1, s_2, \ldots , s_l) \prod_{j=1}^{l} X_{i}(s_j) d\mu(s_j) ~ ; ~ 1 \leq l \leq p.
    \end{split}
\end{equation*}
Similar to $\mathbb{A}^*$ adjoint of the operator $\mathbb{B}_{i}$ denoted as $\mathbb{B}_{i}^* : L^2(S_{2}) \to \mathbb{L}^2 $ is given as
\begin{equation*}
\mathbb{B}_{i}^* g = (B_{i,0}^*g, B_{i,1}^*g, \ldots, B_{i,p}^*g ), \quad \forall~ g \in L^2(S_{2}),
\end{equation*}
where 
\begin{equation*}
    \begin{split}
        B_{i,0}^*g & = g,\\
       B_{i,l}^*g & = g(t) \prod_{j=1}^{l}X_{i}(s_j), \quad 1 \leq l \leq p.
    \end{split}
\end{equation*}
Using the empirical data points $\{(X_{1}, Y_{1}),(X_{2}, Y_{2}),\ldots,(X_{n}, Y_{n})\}$, we consider the estimator $\hat{\beta}$ of the unknown $\beta^*$ which is given as
\begin{equation}
\label{estimator}
    \hat{\beta} := \arg\min_{\beta \in \mathbb{L}^2} \frac{1}{n} \sum_{i=1}^{n} \|Y_{i} - \mathbb{B}_{i}\beta\|^2_{L^2(S_{2})}.
\end{equation}
It can be easily verified that the solution of equation $(\ref{estimator})$ can be given by solving the operator equation
\begin{equation}\label{empirical_operator_equation}
[\mathbb{A}^*\mathbb{A}]_{n} \hat{\beta} = [\mathbb{A}^*Y]_{n},
\end{equation}
where $[\mathbb{A}^*\mathbb{A}]_{n}$ is a matrix of size $(p+1) \times (p+1)$ with entries
\begin{equation*}
    \textcolor{black}{[\mathbb{A}^*\mathbb{A}]_{n}(j,k) = \frac{1}{n} \sum_{i=1}^{n} B_{i,j}^* B_{i,k}, \quad 0 \leq j,k \leq p,}
\end{equation*}
and
\begin{equation*}
    [\mathbb{A}^*Y]_{n} = \left(\frac{1}{n}\sum_{i=1}^{n}B_{i,0}^*Y_{i}, \frac{1}{n}\sum_{i=1}^{n}B_{i,1}^*Y_{i}, \ldots, \frac{1}{n}\sum_{i=1}^{n} B_{i,p}^*Y_{i}\right)^\top.
\end{equation*}
Equation $(\ref{empirical_operator_equation})$ can be seen as a discretized version of operator equation $\mathbb{A}^*\mathbb{A}\beta = \mathbb{A}^*Y$, which is an ill-posed inverse problem. Regularization techniques has been an renounced method to deal with such problems, using which we get that
\begin{equation*}
    \hat{\beta}_{\lambda} = g_{\lambda}([\mathbb{A}^*\mathbb{A}]_{n})[\mathbb{A}^*Y]_{n}.
\end{equation*}
%%%%%%%%%%%%%%%%%%%%%%%%%%%% regularization family %%%%%%%%%%%%%%%%%%%%%%%%%%%%%%%%%
\noindent
Here $g_{\lambda}:[0,\eta ] \to \mathbb{R},\, 0<\lambda \leq \eta$, is the regularization family satisfying the following conditions:
\begin{itemize}
    \item There exists a constant $A>0$ such that
    \begin{equation}
    \label{reg1}
        \sup_{0<\sigma \leq \eta} |\sigma g_{\lambda}(\sigma )| \leq A.
    \end{equation}
    \item There exists a constant $B>0$ such that
    \begin{equation}
    \label{reg2}
        \sup_{0<\sigma \leq \eta} |g_{\lambda}(\sigma )| \leq \frac{B}{\lambda}.
    \end{equation}
    \item The \textbf{qualification} $\nu$ of the regularization family is the largest value such that there exists a constant $\gamma_\nu > 0$ satisfying
    \begin{equation}
        \label{qualification}
        \sup_{0 < \sigma \leq \eta} |r_{\lambda}(\sigma)| \sigma^{\nu} \leq \omega_\nu \lambda^{\nu},
    \end{equation}
    where $r_{\lambda}(\sigma) := 1 - \sigma g_{\lambda}(\sigma)$.
\end{itemize}

\noindent
These properties ensure that $g_{\lambda}$ acts as a stable approximate inverse of $\sigma$, and the qualification measures the rate at which the residual decays with $\lambda$. From $(\ref{reg1})$, we can always choose a constant $D>0$ such that
\begin{equation}
        \label{reg3}
        \sup_{0 < \sigma \leq \eta} |1 - \sigma g_{\lambda}(\sigma)| = \sup_{0 < \sigma \leq \eta} |r_{\lambda}(\sigma)| \leq D.
    \end{equation}
    
\begin{example}[Spectral cut-off] The spectral cut-off method is indexed by the family of functions
    \begin{equation*}
        g_{\lambda}\left(\sigma\right) =
        \begin{cases}
            \frac{1}{\sigma} & \mbox{ if } \lambda \leq \sigma < \infty,\\
            0 & \mbox{ if } 0 \leq \sigma < \lambda.
        \end{cases}
    \end{equation*}
    The qualification of this regularization is infinite as for every $p >0$, the inequality $\left(\ref{qualification}\right)$ holds true.
\end{example}
%-----------------------------------------------------------------
\begin{example}[Tikhonov]
The Tikhonov regularization is indexed by the family of functions $$g_{\lambda}\left(\sigma\right)=\frac{1}{\sigma+\lambda}, \quad \lambda >0. $$ The qualification for this regularization is 1.
\end{example}
Because of its simplicity, the Tikhonov regularization is one of the most popular regularization techniques. 
%-----------------------------------------------------------------
\begin{example}[Landweber iteration] The Landweber regularization is indexed by the family of functions
    \begin{equation*}
        g_{t}\left(\sigma\right) = \sum_{i=1}^{t-1}\left(1-\sigma\right)^i,
    \end{equation*}
    where $\lambda$ is identified as $\frac{1}{t} , t \in \mathbb{N}$.
    The qualification of this regularization is arbitrarily high.
\end{example}
\noindent
We encourage the reader to see \cite{sergei2013,englmartinbook} for more details about the regularization method.

%%%%%%%%%%%%%%%%%%%%%%%%%%%% Section- partition %%%%%%%%%%%%%%%%%%%%%%%%%%%%%%
\section{Main Results}\label{ch_5:sec:main_results}
In this section, we present the convergence analysis of the proposed estimator. We begin by outlining the key assumptions required for the analysis, followed by a few essential lemmas.\\

The convergence rates of any learning algorithm inherently depend on the smoothness of the unknown target function. We make a `prior' assumption of restricting $\beta^*$ to lie within an appropriate subspace of $\mathbb{L}^2$. This condition is generally known as the source condition in learning theory literature.

\subsection{General source condition} The notion of a general source condition in the context of inverse problems has been studied in \cite{howgeneral2008}, where the smoothness of the target function is measured by the rate at which its \textcolor{black}{Fourier coefficients decay, relative to the eigenvalues of the corresponding compact operator and an index function $\varphi$}. In our setting, however, the operator $\mathbb{A}^*\mathbb{A}$ fails to be \textcolor{black}{compact operator}, and therefore the general source condition result of \cite{howgeneral2008} cannot be directly applied to the function-on-function regression framework. Despite this, the structure of $\mathbb{A}^*\mathbb{A}$ enables us to assess the smoothness of any element in $\mathbb{L}^2$ \textcolor{black}{by comparing the decay of its Fourier coefficients to the eigenvalue decay of the operator $\Gamma$}, which corresponds to the restriction of $\mathbb{A}^*\mathbb{A}$ to the subspace $\oplus_{l=0}^p L_{l}^2$. By exploiting this structural property together with the conceptual insights from \cite{howgeneral2008}, we are able to extend the general source condition framework to the function-on-function regression setting. In the specific case of function-on-function polynomial regression, this generalized source condition is formalized through the following assumption.

\begin{assumption}\label{as:1}
 We assume that $\beta^* \in \Omega_{\varphi, R} = \{f\in \mathbb{L}^2 : f = \varphi(\mathbb{A}^*\mathbb{A})v, \|v\|_{\mathbb{L}^2} \leq R \}$,
    where $\varphi :[0,\eta] \to \mathbb{R}~$ is an index function, i.e., $\varphi$ is non-decreasing, continuous and $\varphi(0) = 0$.
\end{assumption}

The operator $\mathbb{A}^*\mathbb{A}$ is not compact, which is a requirement in \cite{howgeneral2008} for formulating the general source condition. Hence, we produce a similar result by utilizing the structure of operator $\mathbb{A}^*\mathbb{A}$.
\begin{prop}
    Let $\mathbb{A}^*\mathbb{A}$ be as defined in section~\ref{ch_5:sec:model_and_preliminaries}. For every $f \in L^2(S_{2})\otimes (\oplus_{l=0}^{p}L_{l}^2)$ and $\epsilon>0$ there is an index function $\varphi$ such that $f = \varphi(\mathbb{A}^*\mathbb{A})\nu$ for some $\nu \in L^2(S_{2})\otimes (\oplus_{l=0}^{p}L_{l}^2)$ with $\|\nu\| \leq (1 +\epsilon)\|f\|$. 
\end{prop}
Let $(\mu_{m}, \phi_{m})_{m \in \mathbb{N}}$ denote the eigenvalue–eigenfunction pairs of the operator $\Gamma$. Then the collection $\{e_{k} \otimes \phi_{m} : k, m \in \mathbb{N}\}$ forms an orthonormal basis for the space $L^2(S_{2}) \otimes (\oplus_{l=0}^{p} L_{l}^2)$. For any $f \in L^2(S_{2}) \otimes (\oplus_{l=0}^{p} L_{l}^2)$, we have the representation

$$f = \sum_{m,k} \langle f, e_{k} \otimes \phi_{m} \rangle_{L^2(S_{2}) \otimes (\oplus_{l=0}^{p} L_{l}^2)} (e_{k} \otimes \phi_{m}),$$
and the sequence $\left( \sum_{k} \langle f, e_{k} \otimes \phi_{m} \rangle_{L^2(S_{2}) \otimes (\oplus_{l=0}^{p} L_{l}^2)}^2 \right)_{m \in \mathbb{N}}$ belongs to $\ell_1$. Using this information, together with the proof techniques developed in \cite{howgeneral2008}, the above proposition can be established in a straightforward manner.

\begin{comment}
    \begin{remark}
    The operator $\mathbb{A}^*\mathbb{A}$ is not compact, which is a requirement in \cite{howgeneral2008} for formulating the general source condition. Hence, we produce a similar result as Theorem~\ref{} by utilizing the structure of operator $\mathbb{A}^*\mathbb{A}$. This source condition extends the general source condition commonly used in the FDA literature, as we do not impose extra constraints on the index function, such as operator monotonicity or Lipschitz continuity. A similar generalized form of the source condition has been explored in \cite{gupta2025revisiting} for the task of learning vectors in Hilbert spaces.
\end{remark}
\end{comment}

\begin{assumption}\label{as:2}
We assume that $$\sup_{\omega \in \Omega }\|X(\omega, \cdot)\|_{L^2(S_{1})}^2 \leq \kappa < \infty.$$
\end{assumption} 
\noindent
Assumption \ref{as:2} demands boundedness over the input stochastic process and is a very standard assumption in the literature \cite{polynomial2023regularization}.
Define $\chi(\omega) = (\chi_{k}(\omega))_{k=0}^{p} \in \oplus_{k=0}^{p}L_{k}^2$, where $\chi_{0}(\omega)= 1$ and $\chi_{k}(\omega)= \prod_{i=1}^{k}X(\omega,s_{i}),~1\leq k \leq p$. Then the operator $\Gamma$ is given as $\Gamma := \mathbb{E}[\chi \boldsymbol{\otimes} \chi]: \oplus_{l=0}^{p} L_{l}^2 \to \oplus_{l=0}^{p} L_{l}^2$, where $\chi \boldsymbol{\otimes} \chi (f) = \langle \chi, f \rangle_{\oplus_{l=0}^{p} L_{l}^2}\chi,~\forall~f \in \oplus_{l=0}^{p} L_{l}^2$. Under Assumption \ref{as:2}, it is easy to see that the operator $\Gamma$ is a trace class operator.
\begin{assumption}\label{as:3}
   For some $b>1$,
\begin{equation*}
    i^{-b}\lesssim \mu_{i} \lesssim i^{-b} \quad \forall  i \in \mathbb{N},
\end{equation*}
where $(\mu_{i},\phi_{i})_{i\in \mathbb{N}}$ is the eigenvalue-eigenfunction pair of operator $\Gamma$.
\end{assumption}

Assumption~\ref{as:3} is commonly used in the learning theory literature to describe the eigenvalue decay of the associated operator, which in our case is $\mathbb{A}^*\mathbb{A}$. However, since $\mathbb{A}^*\mathbb{A}$ is not compact, such an assumption cannot be directly applied to its eigenvalues. To address this, we utilizes the specific structure of $\mathbb{A}^*\mathbb{A}$ and instead place a similar assumption on the eigenvalues of the operator $\Gamma$.\\

\noindent
Assuming $\mu_{i} \lesssim i^{-b},~i \in \mathbb{N}$ ensures the bound on the effective dimension of the operator $\Gamma$ defined as 
$\mathcal{N}(\lambda)= \text{trace}(\Gamma (\Gamma+\lambda I)^{-1})$. It is easy to deduce using $\mu_{i} \lesssim i^{-b},~i \in \mathbb{N}$ that $\mathcal{N}(\lambda) \lesssim \lambda^{-\frac{1}{b}}$.\\

As eigenfunctions $\{\phi_{m}\}_{m \geq 1}$ of operator $\Gamma$ forms an orthonormal basis (ONB) of $\oplus_{l=0}^{p}L_{l}^{2}$, it follows that $\{\phi_{m}\otimes e_{k}\}_{k,m}$ will form a basis for $L^2(S_{2})\otimes (\oplus_{l=0}^{p}L^2_{l})$ where $\{e_{k}\}_{k}$ is an ONB of $L^2(S_{2})$ and $\phi_{m}\otimes e_{k}(u,t) = \phi_{m}(u)e_{k}(t)$ is an element of $L^2(S_{2})\otimes (\oplus_{l=0}^{p}L^2_{l})$. Furthermore, we have that $\mathbb{A}^*\mathbb{A}(f\otimes g) = f \otimes (\Gamma (g)),~\forall~ f \in L^2(S_{2}) \text{ and } g \in \oplus_{l=0}^{p}L_{l}^2$. We use the fact that $\mathbb{L}^2$ is isometrically isomorphic to $L^2(S_{2}) \otimes (\oplus_{l=0}^{p}L_{l}^{2})$.
%%%%%%%%%%%%%%%%%%%%%%%%%%%%%%%%%%%%%%%%%%%%%%%%%%%%%%%%%%%%%%%%%%%%%%%%%%%%%%%%

\noindent

\noindent
\begin{definition}
    Let $q$ be a positive constant. We say that $q$ covers a index function $\varphi$ if there exist a positive constant $c$ such that
    \begin{equation}\label{ch_5:eq_cover_index_function}
        c \frac{\lambda^{q}}{\varphi(\lambda)} \leq \inf_{\lambda \leq \sigma \leq b} \frac{\sigma^{q}}{\varphi(\sigma)}.
    \end{equation}
\end{definition}
In the learning theory literature, it is assumed that the constant $q = \nu$ for the estimation error and $q= \nu-\frac{1}{2}$ for the prediction error where $\nu$ is the qualification of the regularization family. As we are working with an error term which takes care of both the estimation and prediction error, we will assume that $q =\nu-s$, where $0 \leq s \leq \frac{1}{2}$.\\

\noindent
The following lemma helps us with the interplay between the qualification of the regularization family and smoothness of the target function.
\begin{lemma}\label{source_cover_qualification}
Let $\varphi : [0,u] \to \mathbb{R}$ be an index function where $u > 0$. Assume that $\varphi$ is covered by the qualification $\nu$. Then for any $\lambda >0$, we have
    $$\sup_{0 \leq x \leq u} \left|\frac{\varphi(x)}{(x+\lambda)^{\nu}}\right| \leq \max\{1,1/c\}\frac{\varphi(\lambda)}{\lambda^{\nu}}.$$
\end{lemma}
%%%%%%%%%%%%%%%%%%%%%%%%%%%%%%%%%%%%%%%%%%%%%%%%%%%%%%%%%%%%%%%%%%%%%%
\noindent
Next, we present the convergent rate of our proposed estimator to the unknown slope function. The lemmas needed for the proof are included after the proof of the main result.
\begin{theorem}\label{ch_5:upper_bound_main_theorem}
   Suppose Assumption \ref{as:1} and \ref{as:2} hold. Let $\nu$ be the qualification of the regularization family and $\nu \geq 1$. Then for $0 \leq s \leq \frac{1}{2}$, with at least probability $1-\delta$, we have
   \begin{equation*}
       \|(\mathbb{A}^*\mathbb{A})^{s}(\hat{\beta}_{\lambda}- \beta^*)\|_{\mathbb{L}^2} \lesssim \lambda^{s-\frac{1}{2}} \sqrt{\frac{\mathcal{N}(\lambda)}{n}} +  \varphi(\lambda) \lambda^{s}\\
         + {\lambda^{s-\frac{1}{2}}} \left[\frac{2 \tilde{\kappa}^2 \log(\frac{2}{\delta})}{n \sqrt{\lambda}} + \sqrt{\frac{4 \tilde{\kappa}^2\mathcal{N}(\lambda) (\log(\frac{2}{\delta}))^2 }{n}}\right].
   \end{equation*}
\end{theorem}

\begin{proof}
Define $\tilde{\beta} = g_{\lambda}([\mathbb{A}^*\mathbb{A}]_{n})[\mathbb{A}^*\mathbb{A}]_{n} \beta^*$ and start by considering the error term
    \begin{equation*}
    \begin{split}
        \|(\mathbb{A}^*\mathbb{A})^{s}(\hat{\beta}_{\lambda}- \beta^*)\|_{\mathbb{L}^2} = &  \|(\mathbb{A}^*\mathbb{A})^{s}(\hat{\beta}_{\lambda}- \tilde{\beta} + \tilde{\beta} - \beta^*)\|_{\mathbb{L}^2}\\
        \leq & \underbrace{\|(\mathbb{A}^*\mathbb{A})^{s}(\hat{\beta}_{\lambda}- \tilde{\beta})\|_{\mathbb{L}^2}}_{\text{Term-1}} +  \underbrace{\|(\mathbb{A}^*\mathbb{A})^{s}(\tilde{\beta} - \beta^*)\|_{\mathbb{L}^2}}_{\text{Term-2}}.
      \end{split}
    \end{equation*}
\textit{Estimation of Term-1:} By substituting the expressions of $\hat{\beta}_{\lambda}$ and $\tilde{\beta}$, we see
\begin{equation*}
    \begin{split}
        \|(\mathbb{A}^*\mathbb{A})^{s}(\hat{\beta}_{\lambda}- \tilde{\beta})\|_{\mathbb{L}^2} = & \|(\mathbb{A}^*\mathbb{A})^{s} g_{\lambda}([\mathbb{A}^*\mathbb{A}]_{n})([\mathbb{A}^*Y]_{n}- [\mathbb{A}^*\mathbb{A}]_{n}\beta^*) \|_{\mathbb{L}^2}\\
        \leq & \|(\mathbb{A}^*\mathbb{A})^{s} (\mathbb{A}^*\mathbb{A}+\lambda I)^{-\frac{1}{2}}\|_{\mathbb{L}^2 \to \mathbb{L}^2} \|(\mathbb{A}^*\mathbb{A}+\lambda I)^{\frac{1}{2}} ([\mathbb{A}^*\mathbb{A}]_{n}+\lambda I)^{-\frac{1}{2}}\|_{\mathbb{L}^2 \to \mathbb{L}^2}^2\\
        & \times \|([\mathbb{A}^*\mathbb{A}]_{n}+\lambda I)^{\frac{1}{2}}g_{\lambda}([\mathbb{A}^*\mathbb{A}]_{n}) ([\mathbb{A}^*\mathbb{A}]_{n}+\lambda I)^{\frac{1}{2}}\|_{\mathbb{L}^2 \to \mathbb{L}^2}\\
        & \times\|(\mathbb{A}^*\mathbb{A}+\lambda I)^{-\frac{1}{2}}([\mathbb{A}^*Y]_{n}- [\mathbb{A}^*\mathbb{A}]_{n}\beta^*)\|_{\mathbb{L}^2}.
    \end{split}
\end{equation*}
Using Lemma \ref{empirical_bound}, regularization property $(\ref{reg2}), (\ref{reg1})$ and equation $(\ref{ch_5:eqn_powerplay})$, we have
\begin{equation*}
\begin{split}
    \|(\mathbb{A}^*\mathbb{A})^{\frac{1}{2}}(\hat{\beta}_{\lambda}- \tilde{\beta})\|_{\mathbb{L}^2} \lesssim & \lambda^{s-\frac{1}{2}} \|(\mathbb{A}^*\mathbb{A}+\lambda I)^{-\frac{1}{2}}([\mathbb{A}^*Y]_{n}- [\mathbb{A}^*\mathbb{A}]_{n}\beta^*)\|_{\mathbb{L}^2}\\
    \lesssim_{p} & \lambda^{s-\frac{1}{2}} \sqrt{\frac{\mathcal{N}(\lambda)}{n}}.
\end{split}
\end{equation*}
\textit{Estimation of Term-2:} Let $\nu' = \lfloor \nu-s \rfloor$. Consider
\begin{equation*}
    \begin{split}
        \|(\mathbb{A}^*\mathbb{A})^{s}(\tilde{\beta} - \beta^*)\|_{\mathbb{L}^2} = & \|(\mathbb{A}^*\mathbb{A})^{s} (g_{\lambda}([\mathbb{A}^*\mathbb{A}]_{n})[\mathbb{A}^*\mathbb{A}]_{n}\beta^* - \beta^*)\|_{\mathbb{L}^2}\\
        = &  \|(\mathbb{A}^*\mathbb{A})^{s} r_{\lambda}([\mathbb{A}^*\mathbb{A}]_{n}) \varphi(\mathbb{A}^*\mathbb{A})h\|_{\mathbb{L}^2}\\
        \leq & \|(\mathbb{A}^*\mathbb{A})^{s}(\mathbb{A}^*\mathbb{A}+\lambda I)^{-s}\| \|(\mathbb{A}^*\mathbb{A}+\lambda I)^{s}([\mathbb{A}^*\mathbb{A}]_{n}+\lambda I)^{-s}\|_{\mathbb{L}^2}\\
        & \times \|([\mathbb{A}^*\mathbb{A}]_{n}+\lambda I)^{s}r_{\lambda}([\mathbb{A}^*\mathbb{A}]_{n})([\mathbb{A}^*\mathbb{A}]_{n}+\lambda I)^{\nu-s} \|_{\mathbb{L}^2 \to \mathbb{L}^2} \\
        & \times \|([\mathbb{A}^*\mathbb{A}]_{n}+\lambda I)^{-(\nu-s)}\varphi(\mathbb{A}^*\mathbb{A})h\|_{\mathbb{L}^2}.
    \end{split}
\end{equation*}
Now using $(\ref{qualification})$ and the fact that $\nu$ is the qualification of the regularization family, we see
\begin{equation*}
    \begin{split}
     \|(\mathbb{A}^*\mathbb{A})^{s}(\tilde{\beta} - \beta^*)\|_{\mathbb{L}^2} \lesssim_{p} &  \lambda^{\nu} \|([\mathbb{A}^*\mathbb{A}]_{n}+\lambda I)^{-(\nu-s)}\varphi(\mathbb{A}^*\mathbb{A})h\|_{\mathbb{L}^2}\\
        = & \lambda^{\nu} \|([\mathbb{A}^*\mathbb{A}]_{n}+\lambda I)^{-((\nu-s)-\nu')} ([\mathbb{A}^*\mathbb{A}]_{n}+\lambda I)^{-\nu'}\varphi(\mathbb{A}^*\mathbb{A})h\|_{\mathbb{L}^2}\\
        = & \lambda^{\nu}\|([\mathbb{A}^*\mathbb{A}]_{n}+\lambda I)^{-((\nu-s)-\nu')} (\mathbb{A}^*\mathbb{A}+\lambda I)^{(\nu-s)-\nu'}\\ 
        & \times (\mathbb{A}^*\mathbb{A}+\lambda I)^{-((\nu-s)-\nu')} ([\mathbb{A}^*\mathbb{A}]_{n}+\lambda I)^{-\nu'}\varphi(\mathbb{A}^*\mathbb{A})h\|_{\mathbb{L}^2}\\
        \lesssim & \lambda^{\nu}\|(\mathbb{A}^*\mathbb{A}+\lambda I)^{-((\nu-s)-\nu')} ([\mathbb{A}^*\mathbb{A}]_{n}+\lambda I)^{-\nu'}\varphi(\mathbb{A}^*\mathbb{A})h\|_{\mathbb{L}^2},
    \end{split}
\end{equation*}
where the last step follows from $(\ref{ch_5:eqn_powerplay})$. Next by adding and subtracting $(\mathbb{A}^*\mathbb{A}+\lambda I)^{-(\nu-s)}$ to the remaining term, we get
\begin{equation}\label{term2a_term2b}
    \begin{split}
    \|(\mathbb{A}^*\mathbb{A})^{s}(\tilde{\beta} - \beta^*)\|_{\mathbb{L}^2} \lesssim_{p} & \lambda^{\nu'+s} \underbrace{\|(([\mathbb{A}^*\mathbb{A}]_{n}+\lambda I)^{-\nu'}- (\mathbb{A}^*\mathbb{A}+\lambda I)^{-\nu'})\varphi(\mathbb{A}^*\mathbb{A})h \|_{\mathbb{L}^2}}_{\text{Term-2a}}\\
        & + \lambda^{\nu} \underbrace{\|(\mathbb{A}^*\mathbb{A}+\lambda I)^{-(\nu-s)}\varphi(\mathbb{A}^*\mathbb{A})h\|_{\mathbb{L}^2}}_{\text{Term-2b}}.
    \end{split}
\end{equation}
\textit{Bound of Term-2a:} With the use of \cite[Lemma A.3]{gupta2024optimal}, we get
\begin{equation}\label{reducingimp}
    \begin{split}
        & \|([\mathbb{A}^*\mathbb{A}]_{n}+\lambda I)^{-\nu'}\varphi(\mathbb{A}^*\mathbb{A})h - (\mathbb{A}^*\mathbb{A}+ \lambda I)^{-\nu'} \varphi(\mathbb{A}^*\mathbb{A})h\|_{\mathbb{L}^2} \\
        \leq & \|([\mathbb{A}^*\mathbb{A}]_{n}+\lambda I)^{-(\nu'-1)}[([\mathbb{A}^*\mathbb{A}]_{n}+\lambda I)^{-1} - ([\mathbb{A}^*\mathbb{A}]_{n}+\lambda I)^{-1}] \varphi(\mathbb{A}^*\mathbb{A})\|_{\mathbb{L}^2\to \mathbb{L}^2}\\
        & + \left\|\sum_{i=1}^{\nu'-1} ([\mathbb{A}^*\mathbb{A}]_{n}+\lambda I)^{-i} (\mathbb{A}^*\mathbb{A} - [\mathbb{A}^*\mathbb{A}]_{n}) (\mathbb{A}^*\mathbb{A}+ \lambda I)^{-(\nu'+1-i)} \varphi(\mathbb{A}^*\mathbb{A}) \right\|_{\mathbb{L}^2 \to \mathbb{L}^2}\\
        \lesssim & \frac{1}{\lambda^{\nu'-1}} \|([\mathbb{A}^*\mathbb{A}]_{n}+\lambda I)^{-\frac{1}{2}} (\mathbb{A}^*\mathbb{A} - [\mathbb{A}^*\mathbb{A}]_{n}) (\mathbb{A}^*\mathbb{A}+ \lambda I)^{-\frac{1}{2}} \varphi(\mathbb{A}^*\mathbb{A})\|_{\mathbb{L}^2\to \mathbb{L}^2}\\
        \lesssim & \frac{1}{\lambda^{\nu'-1}} \|([\mathbb{A}^*\mathbb{A}]_{n}+\lambda I)^{-\frac{1}{2}} (\mathbb{A}^*\mathbb{A} - [\mathbb{A}^*\mathbb{A}]_{n}) (\mathbb{A}^*\mathbb{A}+ \lambda I)^{-\frac{1}{2}}\|_{\mathbb{L}^2 \to \mathbb{L}^2}\\
        \lesssim_{p} & \frac{1}{\lambda^{\nu'-\frac{1}{2}}} \left[\frac{2 \tilde{\kappa}^2 \log(\frac{2}{\delta})}{n \sqrt{\lambda}} + \sqrt{\frac{ 4 \tilde{\kappa}^2\mathcal{N}(\lambda) (\log(\frac{2}{\delta}))^2 }{n}}\right],
    \end{split}
\end{equation}
where the last step follows from Lemma \ref{ch_5:empirical_vs_original}.\\

\noindent
\textit{Bound of Term-2b:} By using the spectral properties of operator $\mathbb{A}^*\mathbb{A}$, we have
\begin{equation*}
\begin{split}
    \|(\mathbb{A}^*\mathbb{A}+\lambda I)^{-(\nu-s)}\varphi(\mathbb{A}^*\mathbb{A})h\|_{\mathbb{L}^2} = & \sqrt{\sum_{m,k}\langle (\mathbb{A}^*\mathbb{A}+\lambda I)^{-(\nu-s)}\varphi(\mathbb{A}^*\mathbb{A})h, \phi_{m}\otimes e_{k}\rangle_{\mathbb{L}^2}^2}\\
    = & \sqrt{\sum_{m,k}\frac{(\varphi(\mu_{m}))^2\langle h, \phi_{m}\otimes e_{k} \rangle_{\mathbb{L}^2}^2}{(\mu_{m}+\lambda)^{2(\nu-s)}}}\\
    \lesssim & \sup_{i}\left|\frac{\varphi(\mu_{i})}{(\mu_{i}+\lambda)^{(\nu-s)}}\right| \|h\|_{\mathbb{L}^2}.
\end{split}
\end{equation*}
Let us assume that $\varphi$ satisfies equation $(\ref{ch_5:eq_cover_index_function})$ for $p = \nu-s$, then Lemma \ref{source_cover_qualification} concludes that
\begin{equation*}
    \|(\mathbb{A}^*\mathbb{A}+\lambda I)^{-(\nu-s)}\varphi(\mathbb{A}^*\mathbb{A})h\|_{\mathbb{L}^2} \lesssim \varphi(\lambda) \lambda^{s-\nu}.
\end{equation*}
Putting these bounds in $(\ref{term2a_term2b})$ will give us
\begin{equation*}
    \|(\mathbb{A}^*\mathbb{A})^s(\tilde{\beta}-\beta^*)\|_{\mathbb{L}^2} \lesssim_{p} \varphi(\lambda) \lambda^{s}
        + {\lambda^{s-\frac{1}{2}}} \left[\frac{2 \tilde{\kappa}^2 \log(\frac{2}{\delta})}{n \sqrt{\lambda}} + \sqrt{\frac{ 4 \tilde{\kappa}^2 \mathcal{N}(\lambda) (\log(\frac{2}{\delta}))^2 }{n}}\right].
\end{equation*}
Combining all these bounds will conclude that
\begin{equation*}
    \begin{split}
        \|(\mathbb{A}^*\mathbb{A})^s(\hat{\beta}_{\lambda}-\beta^*)\|_{\mathbb{L}^2}  \lesssim_{p} & \lambda^{s-\frac{1}{2}} \sqrt{\frac{\mathcal{N}(\lambda)}{n}} +  \varphi(\lambda) \lambda^{s}\\
         & + {\lambda^{s-\frac{1}{2}}} \left[\frac{2 \tilde{\kappa}^2 \log(\frac{2}{\delta})}{n \sqrt{\lambda}} + \sqrt{\frac{4 \tilde{\kappa}^2\mathcal{N}(\lambda) (\log(\frac{2}{\delta}))^2 }{n}}\right].
    \end{split}
\end{equation*}
\end{proof}

\begin{corollary}\label{ch_5:final_bound}
 Suppose Assumptions \ref{as:1}-\ref{as:3} hold. Let $\psi(x) = \varphi(x) x^{\frac{1}{2}+\frac{1}{2b}}$ and $\nu \geq 1$ be the qualification of the regularization family. Then for $0 \leq s \leq \frac{1}{2}$ and $\lambda = \psi^{-1}(n^{-\frac{1}{2}})$, with at least probability $1-\delta$, we have
\begin{equation*}
    \|(\mathbb{A}^*\mathbb{A})^s(\hat{\beta}_{\lambda}-\beta^*)\|_{\mathbb{L}^2} \lesssim_{p} \varphi(\psi^{-1}(n^{-\frac{1}{2}})) \psi^{-s}(n^{-\frac{1}{2}}).
\end{equation*}
\end{corollary}
\begin{proof}
    Using Assumption \ref{as:3}, we have that $\mathcal{N}(\lambda) \lesssim \lambda^{-\frac{1}{b}}$. Putting this in Theorem \ref{ch_5:upper_bound_main_theorem}, we have
    \begin{equation*}
        \begin{split}
            \|(\mathbb{A}^*\mathbb{A})^s(\hat{\beta}_{\lambda}-\beta^*)\|_{\mathbb{L}^2} \lesssim_{p} \frac{\lambda^{s-\frac{1}{2}-\frac{1}{2b}}}{\sqrt{n}} + \varphi(\lambda)\lambda^s + \frac{\lambda^{s-1}}{n}. 
        \end{split}
    \end{equation*}
Putting $\lambda = \psi^{-1}(n^{-\frac{1}{2}})$ in above equation concludes the result.
\end{proof}
\begin{remark}
    The convergence rates established in Theorem \ref{ch_5:upper_bound_main_theorem} and Corollary \ref{ch_5:final_bound} extend the findings of \cite{polynomial2023regularization} by providing upper bounds for the polynomial regression model with functional responses. Additionally, they eliminate the extra conditions previously imposed on the index function in \cite{polynomial2023regularization}.
\end{remark}

\begin{remark}
Assuming the Hölder source condition, characterized by $\varphi(t) = t^r$ for $r \geq 0$, Corollary~\ref{ch_5:final_bound} implies the following bound:
$$\|(\mathbb{A}^*\mathbb{A})^s(\hat{\beta}_{\lambda}-\beta^*)\|_{\mathbb{L}^2} \lesssim_{p} n^{-\frac{b(r+s)}{1+b+2rb}}.$$ As expected, this convergence rate coincides with the rate established for scalar-valued polynomial regression models under Hölder source conditions in \cite{polynomial2023regularization}. However, while the results in \cite{polynomial2023regularization} are limited to the range $0 \leq r \leq 1$, our analysis holds for all positive values of $r$.
\end{remark}
We now present several key lemmas that are essential in proving the main results of this work. Before coming to these lemmas, it is worth noting that analogous lemmas have been derived in \cite{polynomial2023regularization} for the scalar-on-function polynomial regression model. However, the results in \cite{polynomial2023regularization} rely directly on a Bernstein-type inequality, whose applicability fundamentally depends on the associated operator $\mathbb{A}^*\mathbb{A}$ being Hilbert–Schmidt. Since in our setting $\mathbb{A}^*\mathbb{A}$ does not possess the Hilbert–Schmidt property, this main requirement for the Bernstein-type inequality is not satisfied. Consequently, the arguments from \cite{polynomial2023regularization} cannot be directly extended to our framework, and needs a separate analysis.

\begin{comment}
    \begin{lemma}\cite{pinelis}
\label{pinlis}
Let $\mathcal{H}$ be a Hilbert space and $\xi$ be a random variable taking values in $\mathcal{H}$. Assume that $\|\xi\|_{\mathcal{H}} \leq M$ almost surely. Let $\{\xi_{1},\xi_{2},\ldots,\xi_{n}\}$ be $n$ independent observations of $\xi$.  Then for any $0 < \delta <1$,
    \begin{equation*}
        \left\|\frac{1}{n}\sum_{i=1}^{n}[\xi_{i}-\mathbb{E}(\xi)]\right\|_{\mathcal{H}} \leq \frac{2M log(\frac{2}{\delta})}{n} + \sqrt{\frac{2 \mathbb{E}(\|\xi\|^2_{\mathcal{H}})log(\frac{2}{\delta})}{n}},
    \end{equation*}
with confidence at least $1-\delta$.
\end{lemma}
\end{comment}

\noindent
Observe that $\|([\mathbb{A}^*\mathbb{A}]_{n}-\mathbb{A}^*\mathbb{A})(\mathbb{A}^*\mathbb{A}+\lambda I)^{-1}\|_{\mathbb{L}^2 \to \mathbb{L}^2} \leq \frac{1}{\sqrt{\lambda}}\|([\mathbb{A}^*\mathbb{A}]_{n}-\mathbb{A}^*\mathbb{A})(\mathbb{A}^*\mathbb{A}+\lambda I)^{-\frac{1}{2}}\|_{\mathbb{L}^2 \to \mathbb{L}^2}$. Further we can easily deduce that
    \begin{equation*}
        \begin{split}
         \|(\mathbb{A}^*\mathbb{A}+\lambda I)([\mathbb{A}^*\mathbb{A}]_n+\lambda I)^{-1}\|_{\mathbb{L}^2 \to \mathbb{L}^2}
            = & \|[I-(\mathbb{A}^*\mathbb{A}- [\mathbb{A}^*\mathbb{A}]_{n})(\mathbb{A}^*\mathbb{A}+ \lambda I)^{-1}]^{-1}\|_{\mathbb{L}^2 \to \mathbb{L}^2}\\
             \leq & \frac{1}{1- \|(\mathbb{A}^*\mathbb{A}- [\mathbb{A}^*\mathbb{A}]_{n})(\mathbb{A}^*\mathbb{A}+\lambda I)^{-1}\|_{\mathbb{L}^2 \to \mathbb{L}^2}},
        \end{split}
    \end{equation*}
    provide that $\|(\mathbb{A}^*\mathbb{A}- [\mathbb{A}^*\mathbb{A}]_{n})(\mathbb{A}^*\mathbb{A}+\lambda I)^{-1}\|_{\mathbb{L}^2 \to \mathbb{L}^2} < 1$.\\

\noindent
In the following lemma, we provide a bound on the above defined terms that further will be used to make the bound strictly less than 1.

\begin{lemma}\label{ch_5:empirical_vs_original} Let $\delta >0$ then with probability at least $1-\delta$, we have
    \begin{equation*}
        \|(\mathbb{A}^*\mathbb{A}+\lambda I)^{-\frac{1}{2}}(\mathbb{A}^*\mathbb{A}- [\mathbb{A}^*\mathbb{A}]_{n})(\mathbb{A}^*\mathbb{A}+\lambda I)^{-\frac{1}{2}}\|_{\mathbb{L}^2 \to \mathbb{L}^2} \leq \frac{1}{\sqrt{\lambda}} \left[\frac{2 \tilde{\kappa}^2}{n \sqrt{\lambda}} + \sqrt{\frac{4 \tilde{\kappa}^2\mathcal{N}(\lambda)}{n}}\right]\log(2/\delta).
    \end{equation*}
\end{lemma}

\begin{proof}
\begin{equation*}
    \|(\mathbb{A}^*\mathbb{A}+\lambda I)^{-\frac{1}{2}}(\mathbb{A}^*\mathbb{A}- [\mathbb{A}^*\mathbb{A}]_{n})(\mathbb{A}^*\mathbb{A}+\lambda I)^{-\frac{1}{2}}\|_{\mathbb{L}^2 \to \mathbb{L}^2} \leq \frac{1}{\sqrt{\lambda}} \|(\mathbb{A}^*\mathbb{A}+\lambda I)^{-\frac{1}{2}}(\mathbb{A}^*\mathbb{A}- [\mathbb{A}^*\mathbb{A}]_{n})\|_{\mathbb{L}^2 \to \mathbb{L}^2}.\\
\end{equation*}

\noindent
Let us fix a $f \in \mathbb{L}^2$ such that $\|f\|_{\mathbb{L}^2} = 1$. Then
\begin{equation*}
    \begin{split}
       & \|(\mathbb{A}^*\mathbb{A}+\lambda I)^{-\frac{1}{2}}(\mathbb{A}^*\mathbb{A}- [\mathbb{A}^*\mathbb{A}]_{n})f\|_{\mathbb{L}^2}^2\\
       & \qquad \qquad =  \sum_{k',m'}\langle (\mathbb{A}^*\mathbb{A}+\lambda I)^{-\frac{1}{2}}(\mathbb{A}^*\mathbb{A}- [\mathbb{A}^*\mathbb{A}]_{n})f, \phi_{m'}\otimes e_{k'} \rangle_{\mathbb{L}^2}^2\\
         & \qquad \qquad = \sum_{k',m'} \frac{\langle(\mathbb{A}^*\mathbb{A}- [\mathbb{A}^*\mathbb{A}]_{n})f, \phi_{m'}\otimes e_{k'} \rangle_{\mathbb{L}^2}^2}{\mu_{m'}+\lambda}\\
         & \qquad \qquad = \sum_{k',m'}\left(\sum_{k,j}\langle [A_{k}^*A_{j}-\frac{1}{n}\sum_{i=1}^{n}B_{i,k}^*B_{i,j}]f_{j}, \phi_{m',k}\otimes e_{k'} \rangle_{L^2(S_{2})\otimes L_{k}^2}\right)^2,
    \end{split}
\end{equation*}
where $\phi_{m'} = (\phi_{m',0}, \ldots, \phi_{m',p})$. Next, we consider
\begin{equation*}
    \begin{split}
        & \langle [A_{k}^*A_{j}-\frac{1}{n}\sum_{i=1}^{n}B_{i,k}^*B_{i,j}]f_{j}, \phi_{m',k}\otimes e_{k'} \rangle_{L^2(S_{2})\otimes L_{k}^2}\\
        & \qquad = \int_{\Omega} \langle \langle \chi_{j}, f_{j}\rangle_{L_{j}^2}, e_{k'}\rangle_{L^2(S_2)} \langle \chi_{k}, \phi_{m',k} \rangle_{L^2_{k}} d\mathbb{P}(\omega)- \frac{1}{n}\sum_{i=1}^{n} \langle \langle \chi_{j}^{i}, f_{j}\rangle_{L_{j}^2}, e_{k'}\rangle_{L^2(S_2)} \langle \chi_{k}^{i}, \phi_{m',k} \rangle_{L^2_{k}},
    \end{split}
\end{equation*}
where $\chi_{0}^{i} = 1$ and $\chi_{l}^{i}(s_{1},\ldots, s_{l}) = \prod_{j=1}^{l}X_{j}(s_{j})$.\\

\noindent
Putting things back, we get
\begin{equation*}
    \begin{split}
       & \langle (\mathbb{A}^*\mathbb{A}-[\mathbb{A}^*\mathbb{A}]_{n})f, \phi_{m'}\otimes e_{k'} \rangle_{\mathbb{L}^2}\\
       & \qquad = \int_{\Omega} \langle \langle \chi, f \rangle_{\oplus_{j=0}^{p}L^2_{j}},e_{k'}\rangle_{L^2(S_{2})} \langle \chi, \phi_{m'}\rangle_{\oplus_{k=0}^{p}L^2_{k}} d\mathbb{P}(\omega) - \frac{1}{n}\sum_{i=1}^{n} \langle \langle \chi^{i}, f \rangle_{\oplus_{j=0}^{p}L^2_{j}},e_{k'}\rangle_{L^2(S_{2})} \langle \chi^{i}, \phi_{m'}\rangle_{\oplus_{k=0}^{p}L^2_{k}}\\
       & \qquad = \langle \langle (\mathbb{E}[\chi \otimes \chi]-\frac{1}{n}\sum_{i=1}^{n}\chi^{i}\otimes \chi^{i})\phi_{m'},f \rangle_{\oplus_{j=0}^{p}L_{j}^2},e_{k'}\rangle_{L^2(S_2)}.
    \end{split}
\end{equation*}
 Hence, we have

\begin{equation*}
    \begin{split}
        \|(\mathbb{A}^*\mathbb{A}+\lambda I)^{-\frac{1}{2}}(\mathbb{A}^*\mathbb{A}- [\mathbb{A}^*\mathbb{A}]_{n})f\|_{\mathbb{L}^2}^2 = & \sum_{m'} \frac{\|\langle (\mathbb{E}[\chi \otimes \chi]-\frac{1}{n}\sum_{i=1}^{n}\chi^{i}\otimes \chi^{i})\phi_{m'},f \rangle_{\oplus_{j=0}^{p}L_{j}^2}\|^2_{L^2(S_2)}}{\mu_{m'}+\lambda}\\
        \leq & \sum_{m'} \frac{\|(\Gamma- \Gamma_{n})\phi_{m'}\|^2_{\oplus_{j=0}^{p}L_{j}^2}}{\mu_{m'}+\lambda} \|f\|_{\mathbb{L}^2}^2 \\
        = & \sum_{m'}\frac{\|(\Gamma-\Gamma_{n})\phi_{m'}\|^2_{\oplus_{j=0}^{p}L_{j}^2}}{\mu_{m'}+\lambda} = \|(\Gamma-\Gamma_{n})(\Gamma+\lambda I)^{-\frac{1}{2}}\|^2_{HS}\\
        = & \|(\Gamma+\lambda I)^{-\frac{1}{2}}(\Gamma-\Gamma_{n})\|^2_{HS}.
    \end{split}
\end{equation*}
Taking supremum over all such $f$, we have
\begin{equation*}
    \|(\mathbb{A}^*\mathbb{A}+\lambda I)^{-\frac{1}{2}}(\mathbb{A}^*\mathbb{A}- [\mathbb{A}^*\mathbb{A}]_{n})\|_{\mathbb{L}^2\to \mathbb{L}^2} \leq \|(\Gamma+\lambda I)^{-\frac{1}{2}}(\Gamma-\Gamma_{n})\|_{HS}.
\end{equation*}
From \cite[Lemma 3]{polynomial2023regularization}, for any $\delta \in (0,1)$, with probability at least $1-\delta$, we have that
\begin{equation*}
    \|(\Gamma+\lambda I)^{-\frac{1}{2}}(\Gamma-\Gamma_{n})\|_{HS} \leq \left(\frac{2 \tilde{\kappa}^2}{n \sqrt{\lambda}}+\sqrt{\frac{4 \tilde{\kappa}^2\mathcal{N}(\lambda)}{n}}\right)\log(2/\delta),
\end{equation*}
where $\tilde{\kappa}$ is some positive constant.
\end{proof}

\noindent
Along with Lemma~\ref{ch_5:empirical_vs_original} and Assumption~\ref{as:3}, we can see that
\begin{equation}\label{ch_5:rmk_power}
\|(\mathbb{A}^*\mathbb{A}+\lambda I)([\mathbb{A}^*\mathbb{A}]_{n}+\lambda I)^{-1}\|_{\mathbb{L}^2 \to \mathbb{L}^2} \leq 2,~ \forall~\lambda \geq C_{3}n^{-\frac{b}{1+b}},
\end{equation}
where $C_{3}$ is some positive constant.
\begin{lemma}\cite[Lemma 5.1]{cordes1987}
\label{cordes}
Suppose $T_1$ and $T_2$ are two positive bounded linear operators on a separable Hilbert space. Then
$$\|T_1^pT_2^p\| \leq \|T_1T_2\|^p, \text{ when } 0\leq p \leq 1. $$
\end{lemma}

\noindent
By simply combining $(\ref{ch_5:rmk_power})$ with Lemma \ref{cordes}, we can see that
\begin{equation}\label{ch_5:eqn_powerplay}
    \|(\mathbb{A}^*\mathbb{A}+\lambda I)^{q}([\mathbb{A}^*\mathbb{A}]_n+\lambda I)^{-q}\| \leq 2^{q}\quad \forall~
0 \leq q \leq 1.
\end{equation}
%%%%%%%%%%%%%%%%%%%%%%%%%%%%%%%%%%%%%%%%%%%%%%%%%%%%%%%%%%%%%%%%%%%%%%%%%%%%%%

\begin{remark}
    The next lemma is an analogous of Lemma $4$ \cite{polynomial2023regularization}. The proof given in \cite{polynomial2023regularization} has used the Bernstein type condition on the $\epsilon$, here our proof uses the finite variance of the error term. 
\end{remark}
\begin{lemma}\label{empirical_bound}
For any $\delta >0$, with at least probability $1-\delta$, we have that
    \begin{equation*}
        \|(\mathbb{A}^*\mathbb{A}+\lambda I)^{-\frac{1}{2}}([\mathbb{A}^*\mathbb{A}]_{n}\beta^*- [\mathbb{A}^*Y]_{n})\|_{\mathbb{L}^2} \leq \sqrt{\frac{ \sigma^2 \mathcal{N}(\lambda)}{n \delta}}.
    \end{equation*}
\end{lemma}
\begin{proof}
Recall the definition of $\chi(\omega) = (\chi_{k}(\omega))_{k=0}^{p}$, where $\chi_{0}(\omega) = 1$ and  $\chi_{k}(\omega) = \prod_{j=1}^{k} X(\omega, s_j), 1 \leq k \leq p$.\\
Consider the random variable
$$\xi(\omega) = (\mathbb{A}^*\mathbb{A}+ \lambda I)^{-\frac{1}{2}} \chi(\omega)(Y(\omega, t)- \mathbb{A}\beta^*(\omega, t))= (\mathbb{A}^*\mathbb{A}+ \lambda I)^{-\frac{1}{2}} \chi(\omega) \epsilon(\omega, t)$$ 
taking values in $\mathbb{L}^2$ and 
$$\xi_{i} = (\mathbb{A}^*\mathbb{A}+ \lambda I)^{-\frac{1}{2}}\mathbb{B}_i^*(Y_{i}- \sum_{j=0}^{p}B_{i,j}\beta_{j}).$$
Observe that $\mathbb{E}[\xi(\omega)] = \mathbb{E}[\epsilon(\omega, t)] \cdot \mathbb{E}[(\mathbb{A}^*\mathbb{A}+ \lambda I)^{-\frac{1}{2}} (\boldsymbol{1} \otimes \chi(\omega))]$ = 0 and
\begin{equation*}
    \begin{split}
        \frac{1}{n}\sum_{i=1}^{n}\xi_{i} = & \frac{1}{n}\sum_{i=1}^{n}(\mathbb{A}^*\mathbb{A}+ \lambda I)^{-\frac{1}{2}}(\mathbb{B}_i^*Y_{i} - \sum_{j=0}^{p}\mathbb{B}_i^*B_{i,j}\beta^*_{j})\\
        = & (\mathbb{A}^*\mathbb{A}+\lambda I)^{-\frac{1}{2}}([\mathbb{A}^*\mathbb{A}]_{n}\beta^*- [\mathbb{A}^*Y]_{n}),
    \end{split}
\end{equation*}
where $\boldsymbol{1}$ is the constant function on $S_{2}$ taking value $1$.\\

\noindent
So we can easily see that
\begin{equation}\label{variance}
    \mathbb{E}[\|(\mathbb{A}^*\mathbb{A}+\lambda I)^{-\frac{1}{2}}([\mathbb{A}^*\mathbb{A}]_{n}\beta^*- [\mathbb{A}^*Y]_{n})\|_{\mathbb{L}^2}^2] = \mathbb{E}\|\frac{1}{n}\sum_{i=1}^{n} \xi_{i}\|_{\mathbb{L}^2}^2 = \frac{\mathbb{E}\|\xi\|^2}{n}.
\end{equation}
Next we consider
\begin{equation*}
    \begin{split}
        \mathbb{E}[\|(\mathbb{A}^*\mathbb{A}+\lambda I)^{-\frac{1}{2}}\chi(\omega)\epsilon(\omega,t)\|^2_{\mathbb{L}^2}] = & \mathbb{E}\left[\sum_{m,k} \left\langle (\mathbb{A}^*\mathbb{A}+\lambda I)^{-\frac{1}{2}}\chi(\omega)\epsilon(\omega,t), \phi_{m}\otimes e_{k} \right\rangle_{\mathbb{L}^2}^{2}\right] \\
        = & \mathbb{E}\left[\sum_{m}\frac{\langle \chi(\omega)\epsilon(\omega,t),\phi_{m} \otimes e_{k} \rangle_{\mathbb{L}^2}^{2}}
        {\mu_{m}+\lambda}\right]\\
        = & \mathbb{E} \left[\sum_{m,k}\frac{\langle \chi(\omega), \phi_{m} \rangle_{\oplus_{l=0}^{p}L^2_{l}}^2 \langle \epsilon(\omega,\cdot), e_{k} \rangle_{L^2(S_2)}^2}{\mu_{m}+\lambda}\right]\\
        = & \mathbb{E}\left[\sum_{m}\frac{\langle \chi(\omega), \phi_{m} \rangle_{\oplus_{l=0}^{p}L^2_{l}}^2 \| \epsilon(\omega,\cdot)\|_{L^2(S_2)}^2}{\mu_{m}+\lambda}\right]\\
        = & \sigma^2 \left[\sum_{m}\frac{\langle \Gamma\phi_{m},\phi_{m} \rangle_{\oplus_{l=0}^{p}L^2_{l}}}
        {\mu_{m}+\lambda}\right] = \sigma^2 \mathcal{N}(\lambda).
    \end{split}
\end{equation*}
Putting it in $(\ref{variance})$, we have
\begin{equation*}
    \begin{split}
     \mathbb{E}[\|(\mathbb{A}^*\mathbb{A}+\lambda I)^{-\frac{1}{2}}([\mathbb{A}^*\mathbb{A}]_{n}\beta^*- [\mathbb{A}^*Y]_{n})\|^2] \leq &\frac{\mathcal{N}(\lambda)}{n} \sigma^2.
    \end{split}
\end{equation*}
Now we apply Markov inequality to conclude that
\begin{equation*}
    \|(\mathbb{A}^*\mathbb{A}+\lambda I)^{-\frac{1}{2}}([\mathbb{A}^*\mathbb{A}]_{n}\beta^*- [\mathbb{A}^*Y]_{n})\|_{\mathbb{L}^2} \leq \sigma \sqrt{\frac{\mathcal{N}(\lambda)}{n \delta}}.
\end{equation*}
\end{proof}

\begin{lemma}\cite[Lemma A.6]{balasubramanian2022unified}
\label{supppbound}
    For any $0 < \alpha \leq \beta$,
    $$\sup_{i\in \mathbb{N}}\left[\frac{i^{-\alpha}}{i^{-\beta}+\lambda}\right] \leq \lambda^{\frac{\alpha-\beta}{\beta}}, ~~ \forall ~\lambda >0.$$
\end{lemma}

\begin{lemma}\cite[Lemma A.11]{gupta2024optimal}
\label{seriessum}
    For $\alpha >1$, $\beta >1,$ and $q \geq \frac{\alpha}{\beta}$, we have 
    $$\sum_{i\in \mathbb{N}}\frac{i^{-\alpha}}{(i^{-\beta}+\lambda)^q} \lesssim \lambda^{-\frac{1+\beta q -\alpha}{\beta}}, ~~ \forall ~\lambda >0.$$
\end{lemma}

%%%%%%%%%%%%%%%%%%%%%%%%%%%% Section- partition %%%%%%%%%%%%%%%%%%%%%%%%%%%%%%
\section{Lower Bounds}\label{ch_5:sec:lower_bounds}
In this section, we establish lower bounds for the error term discussed earlier and demonstrate that they align with the upper bounds, thereby validating the optimality of the proposed estimator for $\beta^*$.\\

\noindent
To establish lower bounds, we analyze the divergence between two probability measures, $P_1$ and $P_2$, defined over a measurable space $(\mathcal{X}, \mathcal{A})$. Specifically, we use the \textbf{Kullback-Leibler (KL) divergence}, defined as:

$$\mathcal{K}(P_1, P_2) :=
    \begin{cases}
         \displaystyle \int\limits_\mathcal{X} \log\left(\frac{dP_1}{dP_2}\right) dP_1 & \text{if } P_1 \ll P_2,\\
        +\infty & \text{otherwise},
    \end{cases}$$
where $P_1 \ll P_2$ indicates that $P_1$ is absolutely continuous with respect to $P_2$.

Our approach to deriving lower bounds follows the framework presented in \cite[Chapter 2]{tsyback2009lb}. The idea is to identify $N+1$ elements $\{\theta_0', \ldots, \theta_N'\}$ from the hypothesis space such that the distance between any pair of them is at least $2r$, for some fixed constant $r > 0$. Each $\theta_j'$ corresponds to a probability distribution $P_{\theta_j'}$, and the average KL divergence between $P_{\theta_j'}$ and $P_{\theta_0'}$ should grow at most logarithmically with $N$.

As a result, when $N$ is large, it follows with high probability that any estimator $\hat{\theta}$ will be at least a distance $r$ away from at least one of the $\theta_j'$. For convenience, we restate the key result from \cite[Theorem 2.5]{tsyback2009lb} below.

\begin{theorem}\cite[Theorem 2.5]{tsyback2009lb}
\label{tsyback theorem}
    Assume that $N >2$ and suppose that the hypothesis space $\Theta'$ contains elements $\theta_{0}', \theta_{1}', \ldots, \theta_{N}'$ such that:
    \begin{enumerate}
        \item $d(\theta_{j}', \theta_{k}' ) \geq 2r >0 \quad \forall~ 0\leq j< k \leq N$;
        \item $P_{\theta_{j}'} \ll P_{\theta_{0}'} \quad \forall j =1,2,\ldots,N,$ and
        $$\frac{1}{N} \sum_{j=1}^{N}\mathcal{K}(P_{\theta_{j}'},P_{\theta_{0}'} ) \leq u \log N$$
        for some $0 < u <\frac{1}{8}$. Then
    \end{enumerate}
    $$\inf_{\hat{\theta}}\sup_{\theta\in \Theta'} P_{\theta}(d(\theta,\hat{\theta})\geq r ) \geq \frac{\sqrt{N}}{1+\sqrt{N}}\left(1-2u-\sqrt{\frac{2u}{\log N}}\right),$$
    where the infimum is taken over all estimators  $\hat{\theta}$ of $\theta\in\Theta'$.
\end{theorem}

As evident from Theorem \ref{tsyback theorem}, we require $N+1$ functions with a fixed minimum distance between any pair. The next lemma, famously known as Varshamov-Gilbert Bound, gives the foundation to construct such functions for our analysis.
\begin{lemma}[Varshamov-Gilbert bound \cite{tsyback2009lb}] 
\label{VGbound}
Let $M \geq 8$. Then there exists a subset $\Theta = \{\theta^{(0)},\ldots,\theta^{(N)}\} \subset \{0,1\}^{M}$ such that $\theta^{(0)}=(0,\cdots,0)$,
\begin{equation*}
    H(\theta,\theta^{'}) > \frac{M}{8}, \quad \forall ~~ \theta \neq \theta^{'} \in \Theta,
\end{equation*}
where $\displaystyle H(\theta, \theta^{'}) = \sum_{i=1}^{M}(\theta_{i}-\theta^{'}_{i})^2 $ is the Hamming distance and $N \geq 2^{\frac{M}{8}}$.
\end{lemma}
%%%%%%%%%%%%%%%%%%%%%%%%%%%%%%%%%%%%%%%%%%%%%%%%%%%%%%%%%%%%
Another difficulty in deriving these lower bounds lies in managing the Kullback-Leibler (KL) divergence between two distributions. To address this, we introduce an alternative norm definition for the KL divergence in the context of the polynomial regression model, which will play a crucial role in our analysis.

\subsection{KL divergence for polynomial regression with Gaussian white noise:}
Let us assume that $\epsilon_{i}(t)$ is a Gaussian white noise process with zero mean and constant variance $\sigma^{2}$, independent across $i$ and independent of $(X_{i})_{i=1}^{n}$. Let $P_{\theta}$ and $P_{\theta'}$ be the joint probability measures of $\{(X_{i},Y_{i}) : 1 \le i \le n\}$ corresponding to $\beta^{*}=f_{\theta}$ and $\beta^{*}=f_{\theta'}$, respectively. Then
\begin{equation*}
    \log\!\left(\frac{dP_{\theta}}{dP_{\theta'}}\right) = \log\!\left(\frac{dP_{\theta}(Y|X)}{dP_{\theta'}(Y|X)}\right).
\end{equation*}

\noindent
Since $\epsilon$ is a centred Gaussian white–noise process with covariance 
operator $\sigma^{2}I$ and is independent of $X$, the model
$$Y = \mathbb{A}\beta^{*} + \epsilon$$
immediately implies that $Y|X$ is a Gaussian process with mean
$\mathbb{A}\beta^{*}$ and covariance operator $\sigma^{2}I$. Indeed,
$$\mathbb{E}[Y|X] 
  = \mathbb{A}\beta^* + \mathbb{E}[\epsilon] 
  = \mathbb{A}\beta^*,$$
and covariance operator
$$\operatorname{Cov}(Y|X)
  = \operatorname{Cov}(\epsilon|X)
  = \sigma^{2}I,$$
as $\epsilon$ is independent of $X$.
Thus, conditional on $X$, $Y$ is a Gaussian process with mean 
$\mathbb{A}\beta$ and covariance operator $\sigma^{2}I$.
 Hence $P_{\theta}(Y|X)$ and $P_{\theta'}(Y|X)$ are Gaussian measures with means $\mathbb{A}f_{\theta}$ and $\mathbb{A}f_{\theta'}$, respectively, and common covariance operator $\sigma^{2}I$~\cite{rajput1972gaussian}. Using that $P_{\theta}(Y|X)$, $P_{\theta'}(Y|X)$ are Gaussian measures and Theorem~2.2~\cite{minh2021regularized}, we directly obtain
$$
\mathcal{K}(P_{\theta},P_{\theta'})
= \int \log\!\left(\frac{dP_{\theta}}{dP_{\theta'}}\right)\, dP_{\theta}
= \frac{n}{2\sigma^{2}}
  \left\|(\mathbb{A}^{*}\mathbb{A})^{1/2}(f_{\theta}-f_{\theta'})\right\|_{\mathbb{L}^{2}}^{2}.
$$

\begin{comment}
    \begin{equation*}
    \begin{split}
        \log\left(\frac{P_{\theta}}{P_{\theta'}}\right) = & \log\left(\frac{P_{\theta}(Y|X)}{P_{\theta'}(Y|X)}\right) \\
         = & \frac{1}{2 \sigma^2}\sum_{i=1}^{n} \int_{S_{2}}[(Y_{i}-\mathbb{B}_i f_{\theta'})^2 - (Y_{i}-\mathbb{B}_i f_{\theta})^2]dt\\
            = & \frac{1}{2 \sigma^2}\sum_{i=1}^{n}\int_{S_{2}}[\mathbb{B}_i(f_{\theta}-f_{\theta'})(2Y_{i}-\mathbb{B}_i f_{\theta} - \mathbb{B}_i f_{\theta'})]dt\\
            = & \frac{1}{\sigma^2} \sum_{i=1}^{n}\int_{S_{2}}[(Y_{i}- \mathbb{B}_i f_{\theta})(\mathbb{B}_i(f_{\theta}-f_{\theta'}))]dt + \frac{1}{2 \sigma^2}\sum_{i=1}^{n}\int_{S_{2}}[\mathbb{B}_i (f_{\theta}-f_{\theta'})]^2dt\\
            = & \frac{1}{\sigma^2} \sum_{i=1}^{n}\int_{S_{2}}[(Y_{i}- \mathbb{B}_i f_{\theta})(\mathbb{B}_i(f_{\theta}-f_{\theta'}))]dt + \frac{1}{2 \sigma^2} \sum_{i=1}^{n} \langle \mathbb{B}_{i}(f_{\theta}-f_{\theta'}),\mathbb{B}_{i}(f_{\theta}-f_{\theta'})\rangle_{L^2(S_{2})}.
    \end{split}
\end{equation*}
By the definition of Kullback-Leibler divergence, we have that
\end{comment}

With all the necessary tools at our hand, we are now prepared to present our main result, which provides the lower bounds and demonstrates the optimality of the proposed estimator.
%%%%%%%%%%%%%%%%%%%%%%%%%%%%%%
\begin{theorem}\label{ch_5:lower_bound_main_theorem}
Suppose Assumptions \ref{as:1} and \ref{as:3} hold. Then for $0 \leq s \leq \frac{1}{2}$, we have
\begin{equation*}
        \lim_{a \to 0} \lim_{n \to \infty} \inf_{\hat{\beta}} \sup_{\beta^* \in \mathbb{L}^2} \mathbb{P}\left\{\|(\mathbb{A}^*\mathbb{A})^s(\hat{\beta}-\beta^*)\|_{\mathbb{L}^2} \geq a \varphi(\psi^{-1}(n^{-\frac{1}{2}}))\psi^{-s}(n^{-\frac{1}{2}})\right\} =1,
    \end{equation*}
where the infimum is taken over all possible estimators.
\end{theorem}
\begin{proof}
Define $h(x) = \varphi(x) x^s$ and for given $\epsilon>0$, let $M = \left\lfloor \frac{1}{2} \left(\frac{b_{0}}{h^{-1}\left(\frac{\epsilon}{c_{0}R}\right)}\right)^{\frac{1}{b}}\right\rfloor$ for some constant $0 < c_{0}\leq 1$, which will be specified later. For some $\theta \in \{0,1\}^{M}$, we define $$g_{\theta} = \sum_{k=1}^{M} \frac{\epsilon \theta_{k} M^{-\frac{1}{2}} (\mathbf{1}_{c} \otimes \phi_{k+M})}{\varphi(\mu_{k+M})\mu_{k+M}^s},$$ 
where $\mathbf{1}_{c}$ is an constant function on $S_{2}$ taking value $c=\frac{1}{\sqrt{\mu(S_{2})}}$ with Lebesgue measure $\mu$. Then it is easy to see that
    \begin{equation*}
        \begin{split}
            \|g_{\theta}\|_{\mathbb{L}^2}^2 = & \sum_{k=1}^{M} \frac{\epsilon^2 \theta_{k}^2}{M \varphi^2(\mu_{k+M})\mu_{k+M}^{2s}} \leq \sum_{k=1}^{M} \frac{\epsilon^2}{M \varphi^2(\mu_{k+M})\mu_{k+M}^{2s}}\\
            \leq & \frac{\epsilon^2}{\varphi^2(\mu_{2M})\mu_{2M}^{2s}} \leq \frac{\epsilon^2}{\varphi^2(b_{0}(2M)^{-b})b_{0}^{2s}(2M)^{-2bs}} = \frac{\epsilon^2}{h^2(b_{0}(2M)^{-b})}.
        \end{split}
    \end{equation*}
    Observe that $\|g_{\theta}\|_{\mathbb{L}^2} \leq R$. Let us define 
    $f_{\theta} = \varphi(\mathbb{A}^*\mathbb{A})g_{\theta}$ and it is evident that $f_{\theta} \in \Omega_{\varphi, R}$.\\
By Lemma \ref{VGbound} for $N\geq 2^{\frac{M}{8}}$, we have $\theta^{0},\ldots,\theta^{N} \in \{0,1\}^M$ such that $H(\theta^{i}, \theta^{j}) \geq \frac{M}{8},~\forall~0 \leq i <j \leq N$. Replacing $\theta = \theta^{i}$ in the definition of $f_{\theta}$, we generate $f_{\theta^{i}},~ 0 \leq i \leq N$.
    Consider the error term:
    \begin{equation*}
        \begin{split}
            \|(\mathbb{A}^*\mathbb{A})^s(f_{\theta^{i}}-f_{\theta^{j}})\|^2_{\mathbb{L}^2} = & \sum_{k=1}^{M} \frac{\epsilon^2}{M}(\theta_{k}^{i}-\theta_{k}^{j})\\
            = & \frac{\epsilon^2}{M} H(\theta^{i}, \theta^{j}) \geq \frac{\epsilon^2}{8} ~ \forall ~ 0 \leq i ,j \leq N, ~ i \neq j.
        \end{split}
    \end{equation*}
Next we consider the KL divergence
    \begin{equation*}
        \begin{split}
            \mathcal{K}(P_{\theta^{i}}, P_{\theta^{j}}) = & \frac{n}{2\sigma^2} \|(\mathbb{A}^*\mathbb{A})^{\frac{1}{2}}(f_{\theta^{i}}-f_{\theta^{j}})\|_{\mathbb{L}^2}^2 \\
            = & \frac{n}{2 \sigma^2}\left\|\sum_{k=1}^{M} \frac{\epsilon (\theta^{i}_{k}-\theta^{j}_{k})}{M^{\frac{1}{2}}\varphi(\mu_{k+M})\mu_{k+M}^s} (\mathbb{A}^*\mathbb{A})^{\frac{1}{2}} \varphi(\mathbb{A}^*\mathbb{A}) (\mathbf{1}_{c}\otimes\phi_{k+M}) \right\|^2\\
            = & \frac{n}{2 \sigma^2} \sum_{k=1}^{M} \frac{\epsilon^2(\theta^{i}_{k}-\theta^{j}_{k})^2}{M} \mu_{k+M}^{1-2s} = \frac{n \epsilon^2}{2 \sigma^2 M}\sum_{k=1}^{M}(\theta^{i}_{k}-\theta^{j}_{k})^2 \mu_{k+M}^{1-2s} \\
            \leq & \frac{n \epsilon^2}{2 \sigma^2 M} \mu_{M}^{1-2s}\sum_{k=1}^{M}(\theta^{i}_{k}-\theta^{j}_{k})^2 \leq \frac{n \epsilon^2}{2 \sigma^2} \mu_{M}^{1-2s}\\
            \leq & \frac{n \epsilon^2}{2 \sigma^2} b_{1}^{1-2s} M^{-b(1-2s)} = \frac{n \epsilon^2}{2 \sigma^2} b_{1}^{1-2s} M^{-(b(1-2s)+1)} M.
        \end{split}
    \end{equation*}
    Take $\epsilon = c_{0}R h(\psi^{-1}(n^{-\frac{1}{2}}))$ and 
    observe that $2M \geq \frac{1}{2}\left(\frac{b_{0}}{h^{-1}(\frac{\epsilon}{c_{0}R})}\right)^{1/b}$ and from this we can conclude that
    $$M^{-(b(1-2s)+1)} \leq \frac{4^{b(1-2s)+1}}{b_{0}^{\frac{b(1-2s)+1}{b}}}(\psi^{-1}(n^{-\frac{1}{2}}))^{\frac{b(1-2s)+1}{b}}.$$
    So we get
    \begin{equation*}
            \mathcal{K}(P_{\theta^{i}}, P_{\theta^{j}}) \leq c_{1} n \epsilon^2 M (\psi^{-1}(n^{-\frac{1}{2}}))^{\frac{b(1-2s)+1}{b}},
    \end{equation*}
    where $c_{1} = \frac{4^{b(1-2s)+1}b_{1}^{1-2s}}{2 \sigma^2 b_{0}^{\frac{b(1-2s)+1}{b}}}$. Putting the value of $\epsilon$, we get
    \begin{equation*}
        \begin{split}
             \mathcal{K}(P_{\theta^{i}}, P_{\theta^{j}}) \leq & c_{1} n M c_{0}^2 R^2 (h(\psi^{-1}(n^{-\frac{1}{2}})))^2 (\psi^{-1}(n^{-\frac{1}{2}}))^{\frac{b(1-2s)+1}{b}}\\
             = & c_{1}  n M c_{0}^2 R^2 (\varphi(\psi^{-1}(n^{-\frac{1}{2}})))^2 (\psi^{-1}(n^{-\frac{1}{2}}))^{2s} (\psi^{-1}(n^{-\frac{1}{2}}))^{\frac{b(1-2s)+1}{b}}\\
            = & c_{1}  n M c_{0}^2 R^2 (\varphi(\psi^{-1}(n^{-\frac{1}{2}})))^2 (\psi^{-1}(n^{-\frac{1}{2}}))^{2s+\frac{b(1-2s)+1}{b}}\\
            = & c_{1}  n M c_{0}^2 R^2 (\varphi(\psi^{-1}(n^{-\frac{1}{2}})) (\psi^{-1}(n^{-\frac{1}{2}}))^{\frac{1}{2}+\frac{1}{2b}})^2\\
            = & c_{1}  n M c_{0}^2 R^2 (\psi(\psi^{-1}(n^{-\frac{1}{2}})))^2 = c_{1}   M c_{0}^2 R^2\\
            \leq & \frac{8}{\log 2} c_{1} u \tilde{c}_{0}^2 \log N.
        \end{split}
    \end{equation*}
    In last step we have used $c_{0} = \sqrt{u} \tilde{c}_{0} $ and $N \geq 2^{\frac{M}{8}}$. Now take $\tilde{c}_{0}$ small enough such that $\frac{8 c_{1} \tilde{c}_{0}^2}{\log 2} \leq 1$, then we have
    $$\mathcal{K}(P_{\theta^{i}}, P_{\theta^{j}}) \leq u \log N.$$
Then for $a = \frac{u \tilde{c}_{0}R}{4 \sqrt{2}}$ and by Theorem \ref{tsyback theorem}, we have
$$\inf_{\hat{\beta}} \sup_{\beta \in \mathbb{L}^2} \mathbb{P}\left\{\|(\mathbb{A}^*\mathbb{A})^s(\hat{\beta}-\beta)\|_{\mathbb{L}^2} \geq a \varphi(\psi^{-1}(n^{-\frac{1}{2}}))\psi^{-s}(n^{-\frac{1}{2}})\right\} \geq \frac{\sqrt{N}}{\sqrt{N}+1}\left(1-2u-\sqrt{\frac{2u}{\log N}}\right).$$
Using the fact that $n \to \infty$ implies $N \to \infty$ yields the desired result.
\end{proof}
\vspace{1mm}
\begin{remark}
    The lower bounds established in Theorem \ref{ch_5:lower_bound_main_theorem} align with the convergence rates obtained in corollary \ref{ch_5:final_bound}, thereby confirming the optimality of the proposed estimator. The results of Theorem \ref{ch_5:lower_bound_main_theorem}, when restricted to the scalar response case, establish a lower bound for the convergence rates derived in \cite{polynomial2023regularization}, which was not previously addressed in \cite{polynomial2023regularization}.
\end{remark}

\section*{Acknowledgments}
%%%%%%%%%%%%%%%%%%%%%%%%%%%%% Bibliography %%%%%%%%%%%%%%%%%%%%%%%%%%%%%%%%%%%%%%%%%%
\bibliographystyle{acm}
\bibliography{123}
\end{document}